\documentclass[11pt,twoside,reqno,psamsfonts]{amsart}

\usepackage[left=2.9cm,top=3cm,right=2.9cm]{geometry}              
\usepackage[colorlinks = true, citecolor = black, linkcolor = black, urlcolor = black, pdfstartview=FitH]{hyperref}  
\usepackage[pdftex]{graphicx}
\usepackage[utf8]{inputenc}
\usepackage{microtype, float}
\usepackage{amsmath, verbatim}
\usepackage{amssymb,mathtools}
\usepackage{epstopdf}
\usepackage{epsfig}
\usepackage{mathscinet}
\usepackage[maxbibnames=9]{biblatex}

\usepackage{marginfix}
\usepackage{marginnote}
\usepackage{enumitem}
\bibstyle{abbrv}
\bibliography{references}
\usepackage{hyperref}

\let\oldtocsection=\tocsection
\let\oldtocsubsection=\tocsubsection
\let\oldtocsubsubsection=\tocsubsubsection
\renewcommand{\tocsection}[2]{\hspace{0em}\textbf{\oldtocsection{#1}{#2}}}
\renewcommand{\tocsubsection}[2]{\hspace{1.8em}\oldtocsubsection{#1}{#2}}
\renewcommand{\tocsubsubsection}[2]{\hspace{3em}\oldtocsubsubsection{#1}{#2}}

\DeclareUnicodeCharacter{00A0}{ } 

\numberwithin{equation}{section}

\theoremstyle{plain}

\newtheorem{theorem}{Theorem}[section]
\newtheorem{lemma}[theorem]{Lemma}

\newtheorem{proposition}[theorem]{Proposition}

\theoremstyle{definition}

\newtheorem*{ack}{Acknowledgements}
\newtheorem {definition}[theorem]{Definition}

\newtheorem {example}[theorem]{Example}

\newtheorem*{notation*}{Notation}

\theoremstyle{remark}

\newcommand{\R}{\mathbb{R}}

\newcommand{\N}{\mathbb{N}}

\newcommand{\Z}{\mathbb{Z}}

\renewcommand{\epsilon}{\varepsilon}
\renewcommand{\rho}{\varrho}
\renewcommand{\phi}{\varphi}

\DeclareMathOperator{\Vol}{Vol}

\newcommand{\F}{\mathcal{F}}

\DeclareMathOperator{\poly}{poly}
\DeclareMathOperator{\bigO}{\mathcal{O}}
\DeclareMathOperator{\cl}{cl}
\DeclareMathOperator{\interior}{int}
\DeclareMathOperator{\Mat}{Mat}
\DeclareMathOperator{\Vspan}{span}

\def\XXint#1#2#3{{\setbox0=\hbox{$#1{#2#3}{\int}$}
		\vcenter{\hbox{$#2#3$}}\kern-.5\wd0}}

\title[]{Sharp o-minimality and lattice point counting}

\author{Harrison-Migochi, Andrew \& McCulloch, Raymond}
\address{\noindent Department of Mathematics, University of Manchester,
	Manchester M13 9PL, UK.}
\email{\noindent andrew.harrison-migochi@manchester.ac.uk}

\address{\noindent Department of Mathematics, University of Manchester,
	Manchester M13 9PL, UK\\\newline \indent 
	Heilbronn Institute for Mathematical Research, Bristol, UK.}
\email{\noindent raymond.mcculloch@manchester.ac.uk}
\keywords{Model theory, Pila-Wilkie theorem, sharp o-minimality}
\subjclass[2020]{33E05, 03C64, 11F03} 

\thanks{\noindent ORCID: 0000-0003-4701-5720, ORCID: 0000-0002-0570-4977\\}

\begin{document}
	
	\begin{abstract}
		Let $\Lambda\subseteq\R^n$ be a lattice and let $Z\subseteq\R^{m+n}$ be a definable family in an o-minimal expansion of the real field, $\overline{\R}$. A result of Barroero and Widmer gives sharp estimates for the number of lattice points in the fibers $Z_T=\{x\in\R^n:(T,x)\in Z\}$. Here we give an effective version of this result for a family definable in a sharply o-minimal structure expanding $\overline{\R}$. We also give an effective version of the Barroero and Widmer statement for certain sets definable in $\R_{\exp}$.
	\end{abstract}
	
	\maketitle
	
	\section{Introduction}
	
	In \cite{pila-wilkie} Pila and Wilkie introduced their celebrated point counting theorem on the rational points in a set definable in an o-minimal expansion of the ordered real field. This has led to numerous applications such as various point counting problems. The classical proof of the Pila-Wilkie theorem is non-effective and therefore recent work by Jones and Thomas \cite{effective_pila_wilkie_pfaffian} and Binyamini \cite{density_of_alg_points_noetherian} \cite{point_cointing_for_foliations} as well as Binyamini, Jones, Schmidt and Thomas in \cite{effective_pila_wilkie_subpfaffian} has gone into obtaining an effective proof of the Pila-Wilkie theorem in various cases. 
	
	\noindent In \cite{wilkies_conjecture}, Binyamini, Novikov and Zak introduce the concept of sharp o-minimality (or \#o-minimality) in order to prove Wilkie's Conjecture on $\R_{\exp}$. Sharp o-minimal structures assign to every definable set or definable function two integers, format and degree, which generalise the ideas dimension and degree of semi-algebraic sets. In a \#o-minimal structure, given a finite set of definable sets, $\mathcal{Y}$, the formats and degrees of the sets can then be used to find effective bounds on the format and degree of sets obtained via the application of boolean operations and coordinate projection to definable sets in $\mathcal{Y}$. In particular, effective bounds can be obtained for the number of cells as well as the format and degree of the cells in a cell decomposition compatible with $\mathcal{Y}$. Binyamini, Novikov and Zak \cite{wilkies_conjecture} show that $\R_{\text{rPfaff}}$, the restricted sub-Pfaffian structure over the reals, is \#o-minimal then use this to derive Wilkie's conjecture for $\R_{\exp}$. In their subsequent paper, \cite{sharp_o_minimality}, Binyamini, Novikov and Zak further refine the concept of \#o-minimality. We apply \#o-minimality, as defined in \cite{sharp_o_minimality}, to obtain an effective version of the following theorem by Barroero and Widmer.
	
	\begin{theorem}[\cite{lattice_counting} Theorem 1.3]\label{thm:BW_lattice_counting}
		Let $Z \subseteq \R^{m+n}$ be a definable family in some o-minimal expansion of the real field and suppose that the fibres, $Z_T$, for $T \in \R^m$, are bounded. Then there exists a constant $c_Z \in \R$, depending only on the family, $Z$, such that for all lattices, $\Lambda \subseteq \R^n$
		$$
		\left|\left|Z_T \cap \Lambda\right| - \frac{\Vol(Z_T)}{\det(\Lambda)}\right| \leq c_Z \sum^{n-1}_{j=0} \frac{V_j(Z_T)}{\lambda_1 \ldots \lambda_j},
		$$
		where $V_j(Z)$ is the sum of volumes of the $j$-dimensional orthogonal projections of $Z$ onto the coordinate spaces obtained by setting $n-j$ coordinates to zero and $V_0(Z)$ is taken to be $1$ and $\lambda_1, \ldots, \lambda_n$ are the successive minima of $\Lambda$ with respect to a zero centred unit ball.
	\end{theorem}
	
	\noindent See \cite{geometry_of_numbers} for the definition of successive minima (Chapter VIII) and general lattice theory. Although we do not require the effective Pila-Wilkie, the framework of \#o-minimality provides the necessary tools to prove effective versions of Theorem \ref{thm:BW_lattice_counting}. Theorem \ref{thm:BW_lattice_counting} extends a classical result by Davenport on counting the number of integer points on certain compact subsets of $\R^n$ to a counting theorem for lattice points in families of sets definable in o-minimal expansions of the real field. Barroero and Widmer's counting theorem has been used to obtain asymptotic bounds for the number of algebraic integers of fixed degree and bounded height \cite{counting_algebraic_integers}. In this case, the problem reduces to counting lattice points on a definable family of semi-algebraic sets. Other applications such as \cite{manins_conjecture} and \cite{large_families_of_elliptic_curves_by_conductor} apply point counting to existentially definable sets definable in $\R_{\exp}$. Using our version of Theorem \ref{thm:BW_lattice_counting} in \#o-minimal structures we obtain an effective version for existentially definable sets in $\R_{\exp}$.

	\subsection{Main Results}
	Throughout this paper we use the following notation to denote effectively computable constants:
	\begin{notation*}
		Let $m_1, \ldots, m_k, n \in \N$. Then 
		\begin{itemize}
			\item $\bigO_n(1)$ denotes a constant, $c = c(n)$, such that for some fixed function, $\alpha: \N \to \N$, we have $c \leq \alpha(n)$.
			\item $\poly_n(m_1, \ldots, m_k)$ denotes a constant, $c = c(n, m_1, \ldots, m_k)$, such that \\ $c \leq P_n(m_1, \ldots, m_k)$ for some polynomial $P_n(x_1, \ldots, x_m)$ effectively computable from $n$.
		\end{itemize}
		Different occurrences of $\bigO_n(1)$ and $\poly_n(m_1,\ldots, m_k)$ need not refer to the same function $\alpha$ or polynomial $P_n(x_1, \ldots, x_m)$.
	\end{notation*}
	
	\noindent Our main result is the following effective version of Theorem \ref{thm:BW_lattice_counting} in \#o-minimal structures:
	\begin{theorem}\label{thm:sharp_lattice_count}
		Let $m, n \in \N$ and $Z \subseteq \R^{m + n}$ be a definable family in some \#o-minimal expansion of the real field, such that $Z$ has format $\mathcal{F}$ and degree $D$ and each fibre, $Z_T \in \R^n$, is bounded. Then there exists some constant $c = \poly_\mathcal{F}(D)$ such that for all lattices $\Lambda \subseteq \R^n$,
		$$
		\left| \left|Z_T \cap \Lambda\right| - \frac{\Vol(Z_T)}{\det(\Lambda)} \right| \leq c \sum^{n-1}_{j=0} \frac{V_j(Z_T)}{\lambda_1 \ldots \lambda_j},
		$$
		where $V_j(Z)$ is the sum of volumes of the $j$-dimensional orthogonal projections of $Z$ onto the coordinate spaces obtained by setting $n-j$ coordinates to zero and $V_0(Z)$ is taken to be $1$ and $\lambda_1, \ldots, \lambda_n$ are the successive minima of $\Lambda$ with respect to a zero centred unit ball.
	\end{theorem}
	
	\noindent We use Theorem \ref{thm:sharp_lattice_count} to find an effective version of Theorem \ref{thm:BW_lattice_counting} for $\R_{\exp}$-definable sets with a given existential definition. While we do know by Wilkie's Theorem \cite{Wilkies_Theorem} in $\R_{\exp}$ that all $\R_{\exp}$-definable sets have an existential definition, it is not known how to effectively convert general formulae into existential ones in $\R_{\exp}$ or, indeed, whether or not this is even possible. In our setting given a formula $\varphi$ with format $\F$ and degree $D$ it is not known whether there is an effective bound on the format and degree of the existential formula that it is equivalent to. Hence we require a definition which can be effectively converted into existential form in order to preserve effectivity. We also require the Pfaffian notions of format and degree here. The definitions of Pfaffian format and degree can be found in Section \ref{sec: effec LPC Rexp}.
	
	\begin{theorem}\label{thm:Rexp_lattice_count}
		Let $Z \subseteq \R^{n+m}$ be definable in $\R_{\exp}$ by some existential formula, $\varphi$, of Pfaffian format $\mathcal{F}$ and degree $D$. Suppose the fibre, $Z_T \subseteq \R^n$, is bounded. Then there exists some constant $c = \poly_\mathcal{F}(D)$ such that for all lattices $\Lambda \subseteq \R^n$,
		$$
		\left| \left|Z_T \cap \Lambda\right| - \frac{\Vol(Z_T)}{\det(\Lambda)} \right| \leq c \sum^{n-1}_{j=0} \frac{V_j(Z_T)}{\lambda_1 \ldots \lambda_j},
		$$
		where $V_j(Z)$ is the sum of volumes of the $j$-dimensional orthogonal projections of $Z$ onto the coordinate spaces obtained by setting $n-j$ coordinates to zero and $V_0(Z)$ is taken to be $1$ and $\lambda_1, \ldots, \lambda_n$ are the successive minima of $\Lambda$ with respect to a zero centred unit ball. 
	\end{theorem}
	In the proof of this theorem we use the fact that for an existentially formula we can construct a family of $\R_{\textnormal{Pfaff}}$ definable subsets $\mathcal{Z}'=(Z'_M:M\in\R)$ with uniform format and degree such that for each $T\in\R^n$ we have $Z'_{M,T}=Z_T$ for some $M$. This proof is given in Section \ref{sec: effec LPC Rexp}. In Section \ref{sec: sharp omin} we give the background material from sharp o-minimality that we require. In section \ref{sec: sharp BW} we adapt the argument of Barroero and Widmer in the proof of Theorem \ref{thm:BW_lattice_counting} to the sharp o-minimality setting before proving Theorem \ref{thm:sharp_lattice_count}. We conclude the paper with a discussion on effective o-minimality in Section \ref{sec: eff omin}.

	\section{Sharp O-minimality}\label{sec: sharp omin}  
	\noindent In \cite{sharp_o_minimality}, Binyamini, Novikov and Zak introduce varied strengths of \#o-minimality: presharp, weakly sharp and sharply o-minimal structures. The notion of \#o-minimality in \cite{wilkies_conjecture} corresponds to weakly sharp o-minimality in \cite{sharp_o_minimality}. For the sake of applications we will assume we are working in a \#o-minimal structure in the sense of \cite{sharp_o_minimality}. We recall Definition 1.1 in \cite{sharp_o_minimality}.
	
	\begin{definition}[FD-filtrations]
		Let $\mathcal{R}$ be an o-minimal expansion of the real field, $\overline{\R}$. We say that $\Omega = \{\Omega_{\mathcal{F},D}\}_{\mathcal{F}, D \in \N }$ is an \textit{FD-filtration} on $\mathcal{R}$ if
		\begin{enumerate}
			\item every $\Omega_{\mathcal{F}, D}$ is a collection of definable sets,
			\item $\Omega_{\mathcal{F},D} \subset \Omega_{\mathcal{F} + 1, D} \cap \Omega_{\mathcal{F}, D+1}$ for every $\mathcal{F}, D$, and
			\item every definable set is a member of $\Omega_{\mathcal{F},D}$ for some $\mathcal{F},D$.
		\end{enumerate}
		We call $\mathcal{F}$ the \textit{format} and $D$, the \textit{degree}. We say that a definable function $f$ has \textit{format} $\mathcal{F}$ and \textit{degree} $D$ if its graph, $\Gamma_f$, has format $\mathcal{F}$ and degree $D$.
	\end{definition}
	
	\begin{definition}[\#o-minimality]
		A pair $(\mathcal{R}, \Omega)$ where $\mathcal{R}$ is an o-minimal expansion of $\overline{\R}$ and $\Omega$ is an FD-filtration on $\mathcal{R}$ is called \textit{sharply o-minimal} if for every $\mathcal{F}$ there exists a polynomial $P_\mathcal{F}(x) \in \R[x]$ with non-negative coefficients such that the following axioms are satisfied.
		If $A \in \Omega_{\mathcal{F},D}$ then
		\begin{description}[labelindent=1cm]
			\item[S1] If $A \in \R$, then $A$ has at most $P_\mathcal{F}(D)$ connected components,
			\item[S2] If $A \subseteq \R^l$, then $\mathcal{F} \geq l$,
			\item[S3] If $A \subseteq \R^l$, then $\R \times A, A \times \R \in \Omega_{\mathcal{F} + 1, D}$, and
			\item[S4] If $A \subseteq \R^l$, then $\pi_{l-1}(A), \R^l \setminus A \in \Omega_{\mathcal{F},D}$.
		\end{description}
		If $A_1, \ldots, A_k \subseteq \R^l$ and each $A_i \in \Omega_{\mathcal{F}_i, D_i}$, let $\mathcal{F} = \max_i\{\mathcal{F}_i\}$ and $D = \sum_i D_i$, then 
		\begin{description}[labelindent=1cm]
			\item[S5] $\cup_i A_i \in \Omega_{\mathcal{F},D}$, and
			\item[S6] $\cap_i A_i \in \Omega_{\mathcal{F},D}$.
		\end{description}
		If $P \in \R[x_1, \ldots, x_l]$ then 
		\begin{description}[labelindent=1cm]
			\item[S7] $\{P = 0\} \in \Omega_{l, \deg(P)}$. 
		\end{description}
		The polynomials $P_{\mathcal{F}}$ are taken as part of the data for \#o-minimal structures.
	\end{definition}
	
	\begin{example}
		By \textbf{S7}, the set $\{(x,y,z) \in \R^3: x - y = z^2\} \in \Omega_{3, 2}$. By \textbf{S4}, 
		$\pi_z(\{(x,y,z) : x - y = z^2\}) = \{ (x,y)\in \R^2: \exists z (x - y = z^2)\} \in \Omega_{3,2}$. Therefore, the set $\{(x,y)\in \R^2: x \geq y\} \in \Omega_{3,2}$.
	\end{example}
	\begin{example}
		The set $\{(A,x,y) \in \Mat_{n,m}(\R) \times \R^{n+m} : Ax = y\}$ is the intersection of $n$ sets of the form $\{(A,x,y) \in \Mat_{n,m}(\R) \times \R^{n+m}: \sum_i A_{i, j}x_i = y_j\} \in \Omega_{2(n+m), 2}$. Thus, by \textbf{S5},  $\{(x,y) \in \R^{n+m} : Ax = y\} \in \Omega_{2(n+m), 2n}$.
	\end{example}
	
	\begin{definition}
		We say that a structure $(\mathcal{R}, \Omega)$ where $\mathcal{R}$ is an o-minimal expansion of $\overline{\R}$ with FD-filtration, $\Omega$ is \textit{weakly \#o-minimal} if it satisfies axioms \textbf{S1}-\textbf{S3}, \textbf{S5} and \textbf{S7} of \#o-minimality and the following weaker versions of axioms \textbf{S4} and \textbf{S6}:
		
		\begin{description}[labelindent=1cm]
			\item[W4] If $A \subseteq \R^l$, then $\pi_{l-1}(A), \R^l \setminus A \in \Omega_{\mathcal{F} + 1,D}$.
			\item[W6] If $A_1, \ldots, A_k \subseteq \R^l$ and each $A_i \in \Omega_{\mathcal{F}_i, D_i}$, let $\mathcal{F} = \max_i\{\mathcal{F}_i\}$ and $D = \sum_i D_i$, then $\cap_i A_i \in \Omega_{\mathcal{F} + 1,D}$. 
		\end{description}
	\end{definition}
	
	\begin{definition}
		We say that a structure $(\mathcal{R}, \Omega)$ where $\mathcal{R}$ is an o-minimal expansion of $\overline{\R}$ with FD-filtration, $\Omega$ is \textit{presharp o-minimal} if it satisfies axioms \textbf{S1}-\textbf{S3}, \textbf{W4}, \textbf{S7} and the following weaker versions of axioms \textbf{S5} and \textbf{S6}:\\\\
		If $A_1, A_2 \subseteq \R^l$ and each $A_i \in \Omega_{\mathcal{F}_i, D_i}$, let $\mathcal{F} = \max_i\{\mathcal{F}_i\}$ and $D = \sum_i D_i$, then:
		\begin{description}[labelindent=1cm]
			\item[P5] $A_1 \cup A_2 \in \Omega_{\mathcal{F}+1, D}$, and
			\item[P6] $A_1 \cap A_2 \in \Omega_{\mathcal{F}+1, D}$.
		\end{description}
	\end{definition}
	
	\begin{definition}[\#cell decomposition]
		Let $(\mathcal{R}, \Omega)$ be an o-minimal expansion of $\overline{\R}$ with FD-filtration $\Omega$. We say that $(\mathcal{R}, \Omega)$ has \textit{sharp cell decomposition}, \#cell-decomposition, if for every collection of $k$ definable sets $X_j \subseteq \R^l$, with $X_j \in \Omega_{\mathcal{F},D}$, there exists a cell decomposition of $\R^l$ into $\poly_\mathcal{F}(D, k)$, cells of format $\bigO_\mathcal{F}(1)$ and degree $\poly_\mathcal{F}(D)$ compatible with $X_1, \ldots, X_k$. 
	\end{definition}
	
	\noindent Although it is not known whether all \#o-minimal structures have \#-cell decomposition, Binyamini, Novikov and Zak show that we can always extend a presharp o-minimal structure $(\mathcal{R}, \Omega)$ to a \#o-minimal structure with \#-cell decomposition. More precisely, Binyamini, Novikov and Zak prove the following: 
	
	\begin{definition}[Definition 1.2 \cite{sharp_o_minimality}] (Reductions of FD-filtrations)
		Let $\Omega, \Omega^\prime$ be two FD-filtrations on a structure $\mathcal{S}$. We say that $\Omega$ is \textit{reducible} to $\Omega^\prime$ and write $\Omega \leq \Omega^\prime$ if
		$$
		\Omega_{\mathcal{F}, D} \subseteq \Omega^\prime_{\bigO_{\mathcal{F}}(1), \poly_{\mathcal{F}}(D)} \text{ for all } \mathcal{F}, D \in \N.
		$$
		We say that $\Omega, \Omega^\prime$ are \textit{equivalent} if $\Omega \leq \Omega^\prime$ and $\Omega^\prime \leq \Omega$.
	\end{definition}
	
	\begin{theorem}[Theorem 1.9 \cite{sharp_o_minimality}]\label{thm:presharp_to_sharp}
		Let $(\mathcal{R}, \Omega)$ be a presharp o-minimal structure. Then there exists an FD-filtration, $\Omega^*$, on $\mathcal{R}$ such that $\Omega \leq \Omega^*$ and $(\mathcal{R}, \Omega^*)$ is \#o-minimal with \#cell-decomposition. 
	\end{theorem}
	
	\noindent As a result of Theorem \ref{thm:presharp_to_sharp} we will henceforth assume that \#o-minimal structures we refer to have \#cell-decomposition unless otherwise specified. For our proof of Theorem \ref{thm:sharp_lattice_count} we require the following definability results in \#o-minimal structures analogous to Lemma 3.15 in \cite{lattice_counting}.
	
	\begin{lemma}\label{lem:sharp_family_closure}
		Let $Z \subseteq \R^{m +n}$, with $Z \in \Omega_{\mathcal{F},D}$, be a definable family in some \#o-minimal expansion $(\mathcal{R}, \Omega)$ of the real field. Then
		\begin{enumerate}
			\item $\{(T, x) \in \R^{m + n}: x \in Z_T^c\} \in \Omega_{\bigO_\mathcal{F}(1), \poly_\mathcal{F}(D)}$,
			\item $\{(T,x) \in \R^{m + n}: x \in \interior(Z_T)\} \in \Omega_{\bigO_\mathcal{F}(1), \poly_\mathcal{F}(D)}$,
			\item $\{(T,x) \in \R^{m + n}: x \in \cl(Z_T) \}\in \Omega_{\bigO_\mathcal{F}(1), \poly_\mathcal{F}(D)}$ and
			\item $\{(T,x) \in \R^{m + n}: x \in \partial(Z_T)\} \in \Omega_{\bigO_\mathcal{F}(1), \poly_\mathcal{F}(D)}$,
		\end{enumerate}
		where $Z^c_T$ is the complement of $Z_T$, $\interior(Z_T)$, its interior, $\cl(Z_T)$, its closure and $\partial(Z_T)$, its boundary.
	\end{lemma}
	\begin{proof} The proof follows by straightforward application of the axioms of \#o-minimality. We present the proofs of (1) and (2).
		
		\begin{enumerate}
			\item Let $\pi_x : \R^{m + n} \to \R^m$ be the projection map that forgets the last $n$ coordinates. We have that $\pi_x(Z), Z^c \in \Omega_{\mathcal{F}, D}$. The set $\{(T, x) \in \R^{m + n}: x \in Z_T^c\}$ can be defined as $(\pi_x(Z) \times \R^n) \cap Z^c \in \Omega_{\bigO_\mathcal{F}(1), \poly_\mathcal{F}(D)}$.
			
			\item By the axiom for semi-algebraic sets, \textbf{S7}, we have $$S= 
			\{(T, x, \varepsilon, y) \in \R^{m + n + 1 + n}: \varepsilon > 0, |x-y| \leq \varepsilon \} \in \Omega_{\bigO_\mathcal{F}(1), \bigO_\mathcal{F}(1)}.$$ 
			By applying (1) and permuting variables we obtain,
			$$ U =  \{(T,x,\varepsilon ,y) \in \R^{m + n} \times \R \times \R^n : y \in Z_T^c\} \in \Omega_{\bigO_\mathcal{F}(1), \poly_\mathcal{F}(D)}.$$
			Let $$S^\prime = S \cap U \in \Omega_{\bigO_\mathcal{F}(1), \poly_\mathcal{F}(D)}.$$
			Then we have 
			$$S^\prime = \{(T, x, \varepsilon, y) \in \R^{m+n}\times \R \times \R^n : \varepsilon > 0, |x-y| \leq \varepsilon, y \in Z_T^c \}.$$
			Let $\pi_y: \R^{m + n + 1 +n} \to \R^{m + n + 1}$ denote the projection map that forgets the last $n$-coordinates and let $\pi_{\varepsilon} : \R^{m + n + 1} \to \R^{m + n}$ be the map that forgets the last coordinate.
			The set $\{(T,x) \in \R^{m + n}: x \in \interior(Z_T)\}$ can be written as 
			\begin{align*}
				\pi_\varepsilon((\pi_y(S^\prime))^c) = \{&(T,x)\in \R^{m + n}: \exists \varepsilon > 0 \text{ such that }\\ &\neg \exists y \in \R^n \text{ such that } (T, x, \varepsilon, y) \in S^\prime\}.
			\end{align*}
			It follows then that 
			$$\{(T,x) \in \R^{m + n}: x \in \interior(Z_T)\} \in \Omega_{\bigO_\mathcal{F}(1), \poly_\mathcal{F}(D)}.$$
		\end{enumerate}
	\end{proof}
	
	\noindent We will also require the following sharp version of definable choice.
	
	\begin{proposition}[Proposition 4.1 \cite{sharp_o_minimality}]\label{prop:sharp_definable_choice}
		Let $Z \subseteq \R^{m + n}$ be a definable family such that $Z \in \Omega_{\mathcal{F},D}$ and for all $T \in \R^m$, the fibre $Z_T$ is non-empty. Let $\pi_m: \R^{m + n} \to \R^m$ denote the projection map to the first $m$ coordinates. Then there exists some definable function $f: \pi_m(Z) \to \R^n$ such that for all $T \in \pi_m(Z)$, the image $f(T)$ is in $Z_T$ and $\Gamma_f \in \Omega_{\bigO_\mathcal{F}(1), \poly_\mathcal{F}(D)}$.
	\end{proposition}
	
	
	\section{Lattice Point Counting in Sharp O-minimal Structures}\label{sec: sharp BW}
	
	Barroero and Widmer's proof of Theorem \ref{thm:BW_lattice_counting} in \cite{lattice_counting}  relies on the following theorem of Davenport which counts integer points on certain compact subsets of $\R^n$.
	
	\begin{theorem}[Davenport]
		Let $n$ be a positive integer, and let $Z$ be a compact subset of $\R^n$ such that there exists some constant $h \in \N$ such that
		\begin{enumerate}
			\item the intersection of $Z$ with any line parallel to one of the coordinate axes of $\R^n$ consists of at most $h$ intervals, and
			\item for any $1 \leq j \leq n-1$, the orthogonal projection of $Z$ onto a $j$-dimensional coordinate subspace of $\R^n$ obtained by equating $n - j$ of the coordinates to zero, the same is true
		\end{enumerate}
		then 
		$$
		||Z \cap \Z^n| - \Vol(Z)| \leq \sum^{n-1}_{j=0} h^{n-j} V_j(Z).
		$$
	\end{theorem}
	
	\noindent The strategy used by Barroero and Widmer in \cite{lattice_counting} consists of the following steps:
	\begin{enumerate}
		\item Assume the fibres, $Z_T$, are compact by considering their closures and boundaries.
		\item Show that there is a uniform Davenport constant, $M$ such that for all $T \in \R^m$ and endomorphisms, $\Phi$, of $\R^n$, $$||\Phi(Z_T) \cap \Z^n| - \Vol(\Phi(Z_T))| \leq \sum^{n-1}_{j=0} M^{n-j} V_j(\Phi(Z_T)).$$ Let $\Lambda \subseteq \R^n$ be a lattice and $\lambda_1, \ldots, \lambda_n$ be its successive minima. Further, let $e_1, \ldots, e_n \in \R^n$ denote the usual basis for $\R^n$. Then there exists a basis for $\Lambda$, say $v_1, \ldots, v_n$, such that $|v_i| \leq i \lambda_i$. Note, that for any $Z_T \subseteq \R^n$, if an endomorphism, $\Psi$, of $\R^n$, maps each $v_i$ to $e_i$ then $|\Psi(Z_T) \cap \Z^n| = |Z_T \cap \Lambda|$ and $\Vol(\Psi(Z_T)) = \frac{\Vol(Z_T)}{\det(\Lambda)}$.
		\item Find a constant $c$ such that for all $\Lambda$ and corresponding $\Psi$, $V_j(\Psi(Z_T)) \leq c\frac{V_j(Z_T)}{\lambda_1 \ldots \lambda_j}$ for all $T$.
		\item Use the bounds for the closure and boundary of $Z_T$ to find a bound for $Z_T$.  
	\end{enumerate}
	
	\noindent We use \#o-minimality to obtain effective bounds in steps (2) and (3) depending only on the format and degree of $Z$. For the remainder of this section, fix a \#o-minimal expansion, $(\mathcal{R}, \Omega)$, of $\overline{\R}$ with \#cell-decomposition.
	
	
	\subsection{Volumes of the Closure and Boundary}
	We begin by stating some lemmas from \cite{lattice_counting} that relate the volumes of closure, interior and boundaries of families of definable sets. These will allow us to first prove Theorem \ref{thm:sharp_lattice_count} under the assumption that each fibre $Z_T$ is compact and then use that to prove the more general case.
	
	\begin{lemma}[Lemma 5.3 \cite{lattice_counting}]\label{lem:volume_of_boundary}
		Let $A \subseteq \R^n$ be a bounded definable set. Then $\Vol(\partial(A)) = 0$. In particular, $A$ is measurable and $\Vol(\interior(A)) = \Vol(A) = \Vol(\cl(A))$.
	\end{lemma}
	
	\begin{lemma}[Lemma 5.4 \cite{lattice_counting}]\label{lem:Vj_of_closure}
		Let $Z \subseteq \R^{m+n}$ be a definable family and suppose that the fibres $Z_T$ are bounded. Then for $1 \leq j \leq n-1$, the $j$-dimensional volumes of the orthogonal projections of $Z_T$ on every $j$-dimensional subspace of the coordinate subspace of $\R^n$ exist and are finite. Moreover $V_j(Z_T) = V_j(\cl(Z_T))$.
	\end{lemma}
	
	\subsection{Uniform Davenport Constant}
	
	We introduce the following notation from \cite{lattice_counting}. The first two points of notation were first introduced in the statement of Theorem \ref{thm:BW_lattice_counting} but we include them here to provide a distinction from $V^\prime(Z)$.
	\begin{notation*}
		For a set $Z \subseteq \R^n$, with $1 \leq j \leq n-1$, a $j$-dimensional linear subspace, $X$, of $\R^n$, and linear projection map $\pi: \R^n \to X$ we have that
		\begin{itemize}
			\item $\Vol_j(\pi(Z))$ is the $j$-dimensional volume of $\pi(Z)$,  
			\item $V_j(Z)$ is the sum of volumes of the $j$-dimensional orthogonal projections of $Z$ onto the coordinate spaces obtained by setting $n-j$ coordinates to zero, $V_0(Z)$ is taken to be $1$,
			\item $V^\prime_j(Z)$ is the supremum of the volumes of the orthogonal projections of $Z$ onto $j$-dimensional linear subspaces of $\R^n$.
		\end{itemize}
	\end{notation*}
	
	\begin{notation*}
		Let $I \subseteq \{1, \ldots, n\}$ and let $X \subseteq \R^n$. Let $\Lambda \subseteq \R^n$ be a lattice with successive minima $\lambda_1, \ldots, \lambda_n$ and basis $v_1, \ldots, v_n$ such that $|v_i| < i \lambda_i$. Let $e_1, \ldots, e_n$ be the usual basis for $\R^n$.
		\begin{itemize}
			\item Let $\pi_I: \R^n \to \R^{|I|}$ denote the coordinate projection map onto the span of $\{e_i : i \in I\}$.
			\item Let $\Sigma_I = \Vspan\{v_i: i \in I\}$ and let $\Sigma^\prime_I = \Vspan\{v_i: i \not \in I\}$. We let $X^I$ denote the set defined as follows:
			$$
			\{x \in \Sigma_I: x + y \in X \text{ and } y \in \Sigma^\prime_I\}.
			$$
		\end{itemize}
	\end{notation*}
	
	\noindent To obtain an effective uniform Davenport constant we adapt the proof of Lemma 4.1 in \cite{lattice_counting} to show that the constant obtained is effective.
	
	\begin{lemma}\label{lem:uniform_davenport_constant}
		Let $Z \subseteq \R^{m+n}$ be a definable family such that $Z$ has format less than or equal to $\mathcal{F}$ and degree bounded above by $D$ and for all $T \in \R^m$, the fibre $Z_T$ is bounded. Then there exists $M= \poly_\mathcal{F}(D)$ such that for every $T \in \R^m$ and every endomorphism $\Psi$ of $\R^n$, the constant $M$ is a Davenport constant for $\Psi(Z_T)$. That is
		$$
		\left|\left|\Psi(Z_T)\cap \Z^n \right| - \Vol(\Psi(Z_T))\right| \leq \sum_{j = 0}^{n-1}M^{n-j}V_j(\Psi(Z_T)).   
		$$
	\end{lemma}
	\begin{proof}
		Identify the endomorphism $\Psi$ with its matrix in $\R^{n^2}$. Let 
		\begin{equation}\label{eqn:W}
			W = \{(\Psi, T, x) \in \R^{n^2 + m + n}: x \in \Psi(Z_T)\}.
		\end{equation}
		Let $V = \{(\Psi, T, x, y) \in \R^{n^2 + m + n + n}: x = \Psi(y)\}$. As $V$ is defined by at most $n$ polynomials of degree $2$, it follows that $V \in \Omega_{\bigO_\mathcal{F}(1), \bigO_\mathcal{F}(1)}$. Thus $W$, which is the projection to the first $n^2 + m + n$-coordinates of $V \cap (\R^{n^2 + n} \times Z)$, is definable and is in $\Omega_{\bigO_\mathcal{F}(1), \poly_\mathcal{F}(D)}$.
		\\\\
		Let $I$ be a proper subset if $\{1, \ldots,n\}$. Let $\pi^\prime_I$ be the map on $\R^{n^2 + m + n}$ given by $(\Psi, T, x) \mapsto (\Psi, T, \pi_I(x))$. Note that $\pi^\prime_I(W) \in \Omega_{\bigO_\mathcal{F}(1), \poly_\mathcal{F}(D)}$. 
		\\\\ 
		For $i_0 \in I$, we can parametrise a line in $\pi_I(\R^n)$ parallel to $e_{i_0}$ with a set of real parameters $(l_i)_{i \in I \setminus \{i_0\}} \in \R^{|I| - 1}$ by considering the set $\{(x_i)_{i \in I} \in \pi_I(\R^n): x_i = l_i, \ i \in I \setminus \{i_0\}$. Consider the family of subsets of $\pi^\prime_I(W)$ intersected with such lines:
		\begin{align*}
			L(I, i_0) := \{&((l_i)_{i \in I \setminus \{i_0\}}, \Psi, T, x) \in \R^{|I|-1} \times \R^{n^2 + m + n} : (\Psi, T, x) \in \pi^\prime_I(W),\\
			& l_i = x_i \text{ for } i \in I \setminus \{i_0\}\}. 
		\end{align*}
		Then 
		\begin{align*}
			L(I, i_0) = (\R^{|I|-1} \times \pi^\prime_I(W)) \cap \{&((l_i)_{i \in I \setminus \{i_0\}}, \Psi, T, x) \in \R^{|I|-1} \times \R^{n^2 + m + n} : \\&l_i = x_i \text{ for } i \in I \setminus \{i_0\}\} \in \Omega_{\bigO_{\mathcal{F}}(1), \poly_\mathcal{F}(D)}.
		\end{align*}
		
		\noindent By \#cell-decomposition there is a uniform upper bound $M(I, i_0) = \poly_\mathcal{F}(D)$ on the number of connected components of fibres $L(I, i_0)_{((l_i), \Psi, T)}$. Let $M = \poly_\mathcal{F}(D)$ be the maximum of all $M(I,i_0)$. By construction, $M$ is a uniform Davenport constant for the $\Psi(Z_T)$.
	\end{proof}
	
	\subsection{Uniform Volume Bounds}
	In \cite{lattice_counting}, Barroero and Widmer show that for any endomorphism $\Psi$ of $\R^n$, the sum of the volumes $j$-dimensional coordinate projections of $\Psi(Z_T)$ is bounded by some uniform - across all endomorphisms, $\Psi$ and parameters $T$ - constant times $V_j(Z_T)$. We show that in a \#o-minimal setting, this constant is effective.
	
	\begin{lemma}[Lemma 2.2 \cite{lattice_counting}]\label{lem:endomorphism_image_bound}
		Suppose that $C \subseteq R^n$ is compact. Then for every $j = 1 , \ldots, n-1$,
		$$
		V_j(\Psi(C)) \leq \sum_{|I| = j}\frac{2^j}{B_j} \frac{\Vol_j(C^I)}{\lambda_1 \ldots \lambda_j},
		$$
		where $B_j$ is the volume of the j-dimensional unit ball and for $I \subseteq \{1, \ldots, n\}$, is $C^I$ be the orthogonal projection of $C$ to the subspace of $\R^n$ spanned by $\{e_i: i \in I\}$.
	\end{lemma}
	
	\noindent In \cite{lattice_counting} Barroero and Widmer use the $j$-dimensional Hausdorff measure on $\R^n$ ($j$-Hausdorff measure) to bound certain volumes. As we gain no improvement on their measurability results from the \#o-minimal setting we state only what is directly applicable to the proof of Theorem \ref{thm:sharp_lattice_count}. For greater detail see \cite{lattice_counting}. Let $\mathcal{H}^j$ denote the $j$-Hausdorff measure on $\R^n$ for $n > j$ and let $\mathcal{L}^j$ denote the Lebesque measure on $\R^j$.
	
	\begin{definition}
		Let $X \subseteq \R^n$. We say that $X$ is \textit{$j$-rectifiable} for $j \leq n$ if there exists some Lipschitz function, $R^j \to X$, mapping some bounded subset of $\R^j$ onto $X$. Moreover, $X$ is \textit{$(\mathcal{H}^j, j)$-rectifiable} if there exist countably many $j$-rectifiable sets whose union is $\mathcal{H}^j$-almost $X$ and $\mathcal{H}^j(X) < \infty$. 
	\end{definition}
	
	\begin{proposition}[Proposition 5.2 \cite{lattice_counting}] \label{prop:coincidence_of_dimension}
		Suppose $X \subseteq \R^n$ is non-empty and definable. Then $\dim(X)$ coincides with its Hausdorff dimension. Moreover, if $\dim(X) = d$ and $X$ is bounded, then $X$ is $j$-Hausdorff measurable for every $j$ with $d \leq j \leq n$. Finally, $\mathcal{H}^d(X) < \infty$ and $\mathcal{H}^j(X) = 0$ for $j > d$.
	\end{proposition}
	
	\noindent Barroero and Widmer \cite{lattice_counting} make use of the $C^r$-parametrisation Lemma from the proof of the Pila-Wilkie Theorem in order to obtain a finite number of Lipschitz functions to show $(\mathcal{H}^j, j)$-rectifiability. In a \#o-minimal structure we have an alternate version of $C^r$-parametrisation which provides an effective bound depending only on the format and degree of $Z$ on the number of Lipschitz functions needed to almost cover $Z$. However, this effective bound does not have an effect on the proof of Theorem \ref{thm:sharp_lattice_count}.
	
	\noindent In Lemma 5.8 \cite{lattice_counting}, Barroero and Widmer obtain constants for each projection $\pi_I$. This constant relies on cell decomposition. For our purposes, we require only a single uniform constant across all projections, $\pi_I$, and apply \#cell-decomposition to obtain such a uniform and effective constant. The following lemma is an effective version of Lemma 5.8 in \cite{lattice_counting}.

	\begin{lemma} \label{lem:sharp_bound_on_Hj}
		Let $S \subseteq \R^{p + n}$, with $S \in \Omega_{\mathcal{F},D}$, be a definable family whose fibres $S_a \in \R^n$ are bounded and of dimension at most $j \geq 1$. Then there exists  $E = \poly_\mathcal{F}(D)$ such that 
		$$ \mathcal{H}^j(S_a) \leq \sum_{|I| = j} E \Vol_j(\pi_I(S_a))$$
		for every $a \in \R^p$.
	\end{lemma}
	
	\begin{proof}
		If $S = \varnothing$, the claim is true. For $a \in \R^p$, if $\dim(S_a) < 1$, then $\mathcal{H}^j(S_a) = 0$. Thus we assume that $\dim(S_a) > 0$. Therefore by \ref{prop:coincidence_of_dimension}, $S_a$ is $(\mathcal{H}^j,j)$-rectifiable. Hence 
		$$\mathcal{H}^j(S_a) \leq \sum_{|I| = j} \int N(\pi_I| S_a, y)d\mathcal{L}^jy$$
		for every $a \in \R^p$.
		\\\\
		Let
		$$ R = \{(a,y,x) \in \R^{p+j+n}: (a,x) \in S, y = \pi_I(x)\}.$$
		Since the graph of $\pi_I$ is in $\Omega_{\bigO_n(1), \bigO_n(1)}$, it follows that, $R \in \Omega_{\bigO_\mathcal{F}(1), \poly_\mathcal{F}(D)}$.
		\\\\
		For every $(a,y)$ we have $|R_{(a,y)}| = |\pi^{-1}(y) \cap S_a| = N(\pi_I| S_a,y)$. By \#cell decomposition there is some a uniform upper bound $E = \poly_\mathcal{F}(D)$ for the number of connected components of the fibres of $R_{(a,y)}$. That is if $\dim(R_{(a,y)}) = 0$, then $|R_{(a,y)}| \leq E$.
		\\\\
		For fixed $a \in \R^p$. The restriction $\pi_{I|S_a}: S_a \to \R^j$ is a definable map. Thus, we get 
		$$P = \{y \in \R^j: \dim(\pi^{-1}(y)) \cap S_a \geq 1\}$$
		is definable and
		$$\dim(P) \leq \dim S_a - 1 \leq j-1.$$
		Hence $P$ has measure $0$ in $\R^j$. Let $Q= \pi(S_a) \setminus P$. The set $Q$ is definable and is the set of $y$ such that $\dim(R_{(a,y)}) = 0$. Therefore
		$$\int N(\pi_I|S_a,y)d\mathcal{L}^jy = \int_Q|R_{(a,y)}|d\mathcal{L}^jy \leq \int_Q Ed\mathcal{L}^jy = E \Vol_j(\pi_I(S_a)). $$
	\end{proof}
	
	\noindent The following is Lemma 2.4 in \cite{lattice_counting}. The explicit constant is taken directly from Barroero and Widmer's proof of the lemma. We emphasise the explicit nature of the constant to show that for a definable compact set $C$, we have that $\Vol_j(C)$ is at most $ cV^\prime(C)$ for some constant $c = \bigO_{\mathcal{F}}(1)$.
	\begin{lemma}[Lemma 2.4 \cite{lattice_counting}]\label{lem:orthogonal_projection_bound}
		Suppose that $C \subseteq \R^n$ is compact. Then for any $j = 1, \ldots, n-1$ and $I \subseteq \{1, \ldots, n\}$ with $|I| = j$
		$$
		\Vol_j(C^I) \leq \left(j^{3/2}\frac{n!2^n}{B_n}\right)^j V^\prime_j(C).
		$$
	\end{lemma}
	
	\begin{lemma}[Lemma 5.1 \cite{lattice_counting}] \label{lem:measure_of_image_under_endomorphism}
		Suppose $A \subseteq \R^n$ for $1 \leq j \leq n$ and suppose that $A$ is $j$-Hausdorff measurable. Furthermore, let $\psi: \R^n \to \R^n$ be an endomorphism. Then $\mathcal{H}^j(\Psi(A)) \leq ||\Psi||^j_{\textnormal{op}} \mathcal{H}^j(A)$. Moreover if $\Psi$ is an orthogonal projection, then $\mathcal{H}^j(\Psi(A)) \leq \mathcal{H}^j(A)$ and if $\Psi \in O_n(\R)$ then $\mathcal{H}^j(\Psi(A)) \leq \mathcal{H}^j(A)$. Here $||\Psi||^j_{\textnormal{op}}$ is the operator norm on $\Psi$, that is $||\Psi||^j_{\textnormal{op}} = \inf\{c \in \R: ||\Psi(v)|| < c \textnormal{ for all } v \in \R^n\}$.
	\end{lemma}
	
	\noindent In Lemmas 6.1 and 6.2 in \cite{lattice_counting}, Barroero and Widmer construct an auxiliary set, $Z^\prime$, in order to bound $V_j^\prime(Z_T)$ by $cV_j(Z_T)$ for some constant $c$. We adapt Lemma 6.2 \cite{lattice_counting} to show that in the \#o-minimal setting the set $Z^\prime$ has format and degree effectively bounded in terms of the format and degree of $Z$ and thus the constant obtained in Lemma 6.2 of \cite{lattice_counting} is effective. The following two lemmas are our analogues of Lemmas 6.2 and 6.1 in \cite{lattice_counting} respectively.
	
	\begin{lemma}\label{lem:auxilliary_set}
		Let $Z \subseteq \R^{m + n}$, with $Z \in \Omega_{\mathcal{F}, D}$ be a definable family such that the fibres $Z_T$ are bounded. Let $j \in \{1, \ldots, n\}$. There exists some definable $Z^\prime \subseteq \R^{n^2 + m + n}$, with $Z^\prime \in \Omega_{\bigO_\mathcal{F}(1), \poly_\mathcal{F}(D)}$, such that 
		\begin{itemize}
			\item $\dim(Z^\prime_{(\Phi, T)}) \leq j$,
			\item $Z^\prime_{(\Phi, T)} \subseteq Z_T$ for all $(\Phi, T) \in \R^{n^2 + m}$, and
			\item $V^\prime_j(Z_T) \leq \sup_{\Phi \in O_n(\R)}\mathcal{H}^j(Z^\prime_{(\Phi,T)})$
			for every $T\in \R$.
		\end{itemize} 
		
	\end{lemma}
	\begin{proof}
		Let $S = \{(\Phi, T,x) \in \R^{n^2 + m + n} : \Phi \in O_n(\R), y \in \Phi(Z_T)\}$ be $W$ from \eqref{eqn:W} in the proof of Lemma \ref{lem:uniform_davenport_constant} intersected with $O_n(\R) \times \R^{m+n}$. It follows that since $O_n(\R)$ is defined by $n$ degree $2$ polynomial equations $O_n(\R) \times \R^{m+n} \in \Omega_{\bigO_n(1), \bigO_n(1)}$ and therefore $S \in \Omega_{\bigO_\mathcal{F}(1),\poly_\mathcal{F}(D)}$. Further, note that $S_{(\Phi, T)} = Z_T$ for every $(\Phi, T) \in O_n \times \R^m$.
		\\\\
		Let $\pi: \R^{n^2 + m +n} \to \R^{n^2 + m + j}$ be the projection that forgets the last $n-j$ coordinates. By Proposition \ref{prop:sharp_definable_choice} there exists some definable function $f: \pi(S) \to \R^{n-j}$ such that the graph of $f$, denoted $\Gamma(f)$,  is a subset of $S$ and $\Gamma(f) \in \Omega_{\bigO_\mathcal{F}(1), \poly_\mathcal{F}(D)}$.
		\\\\
		We claim that $\dim(\pi(S)_{(\Psi, T)}) = \dim(\Gamma(f)_{(\Psi, T)})$. Let $F: \pi(S) \to \Gamma(f)$ be the map $(\Psi, T, x) \mapsto (\Psi, T, z, f(\Psi, T, z))$. Note that by construction $F$ is a bijection and is the inverse of $\pi|_{\Gamma(f)}$. For fixed $(\Psi, T)$, the bijection $F$ induces a bijection $x \mapsto f(\Psi, T, x)$ from $\pi(S)_{(\Psi, T)}$ to $\Gamma(f)_{(\Psi, T)}$. Thus the claim holds.
		\\\\
		Let 
		$$Z^\prime := \{(\Psi, T, x) \in \R^{n^2 + m + n}: \Psi \in O_n(\R), \Psi(x) \in \Gamma(f)_{(\Psi,T)}\}.$$
		The set $Z^\prime$ can be written a projection of the intersection of $\Gamma(f)\times \R^n$ and $\{(\Psi, T, y , x) \in \R^{n^2 +m + n + n}: \Psi(y) = x\} \in \Omega_{\bigO_n(1), \bigO_n(1)}$. Thus $Z^\prime \in \Omega_{\bigO_\mathcal{F}(1), \poly_\mathcal{F}(D)}$.
		\\\\
		We have that 
		$$ \Psi(Z^\prime_{(\Psi,T)}) = \Gamma(f)_{(\Psi,T)}$$
		for every $(\Psi, T) \in O_n(\R) \times \R^m$. If $\Psi \in \R^{n^2} \setminus O_n(\R)$, we have $Z^\prime_{(\Psi, T)} = \varnothing$. Therefore $\dim(Z^\prime_{(\Phi, T)}) \leq j$ and $Z^\prime_{(\Psi, T)} \subseteq Z_T$.
		\\\\
		We now need only show that 
		$$ V^\prime_j(Z_T) \leq \sup_{\Phi \in O_n(\R)}\mathcal{H}^j(Z^\prime_{(\Phi,T)}).$$
		Recall that $V^\prime_j(Z_T)$ is the supremum of the volumes of the orthogonal projections of $Z_T$ onto a $j$-dimensional subspace. Let $\Sigma$ be a $j$-dimensional subspace of $\R^n$ and let $\pi_\Sigma$ be the orthogonal projection of $\R^n$ onto $\Sigma$. Consider the subspace of $\R^n$ spanned by the vectors $\{e_1, \ldots, e_j\}$ and let $\tilde{\pi}$ be the orthogonal projection onto this subspace. By taking some $\Psi \in O_n(\R)$ such that $\Psi$ maps an orthonormal basis for $\Sigma$ to the set $\{e_1, \ldots, e_j\}$ we obtain $\Psi \circ \pi_\Sigma = \tilde{\pi} \circ \Psi$. That is, by Lemma \ref{lem:measure_of_image_under_endomorphism}, 
		$$
		\Vol_j(\pi_\Sigma(Z_T)) = \Vol_j(\Psi(\pi_\Sigma(Z_T))) = \Vol_j(\tilde{\pi}(\Psi(Z_T))) = \Vol_j(\tilde{\pi}(S_{(\Psi, T)})).
		$$
		Thus
		$$
		V^\prime_j(Z_T) = \sup_\Sigma \Vol_j(\pi_\Sigma(Z_T)) \leq \sup_{\Psi \in O_n(\R)} \Vol_j(\tilde{\pi}(S_{(\Psi, T)})).
		$$
		Further, note that from the definition of $\tilde{\pi}$, the image $\tilde{\pi}(S_{(\Psi, T)}) = \pi(S)_{(\Psi, T)}$.
		Thus we have that 
		$$
		\tilde{\pi}(S_{(\Psi, T)}) = \tilde{\pi}(\Gamma(f)_{(\Psi, T)}).
		$$
		Thus by Lemma \ref{lem:measure_of_image_under_endomorphism}, we obtain
		$$
		\Vol_j(\tilde{\pi}(S_{(\Psi, T)}) = \mathcal{H}^j(\tilde{\pi}(S_{(\Psi, T)})) \leq \mathcal{H}^j(\Gamma(f)_{(\Psi, T)}).
		$$
		But we have that, by Lemma \ref{lem:measure_of_image_under_endomorphism},
		$$
		\mathcal{H}^j(\Gamma(f)_{(\Psi, T)}) = \mathcal{H}^j(Z^\prime_{(\Psi,T)})
		$$
		for every $(\Psi, T) \in O_n(\R) \times \R^m$.
	\end{proof}
	
	\begin{proposition}\label{prop:bound_V_prime}
		Let $Z \subseteq \R^{m+n}$ be a definable family such that $Z \in \Omega_{\mathcal{F},D}$ and the fibres $Z_T$ are bounded. Let $j \in \N$, such that $j \leq n - 1$. Then there exists a constant $K = \poly_\mathcal{F}(D)$ such that for all $T \in \R^m$
		$$
		V^\prime_j(\cl(Z_T)) \leq K V_j(Z_T).
		$$
	\end{proposition}
	
	\begin{proof}
		By Lemma \ref{lem:Vj_of_closure} we have that $V_j(Z_T) = V_j(\cl(Z_T))$. Let $Z^\prime$ be as in Lemma \ref{lem:auxilliary_set}. By Lemma \ref{lem:sharp_bound_on_Hj} there exists some $E = \poly_\mathcal{F}(D)$ such that 
		$$
		\mathcal{H}^j(Z^\prime_{(\Psi, T)}) \leq \sum_{|I| = j} E\Vol_j(\pi_I(Z^\prime_{(\Psi, T)}))
		$$
		for every $(\Psi, T) \in \R^{n^2 + m}$.
		\\\\
		Let $\pi_{C^I}$ be the orthogonal projection map from $\R^n$ onto the subspace $C^I$ spanned by $\{e_i: i \in I\}$. We have
		$$
		\Vol_j(\pi_I(Z^\prime_{(\Psi, T)})) = \Vol_j(\pi_{C^I}(Z^\prime_{(\Psi, T)}))
		$$
		for every $(\Psi, T)\in \R^{n^2 + m}$. Therefore, since $Z^\prime_{(\Psi, T)} \subseteq Z_T$, 
		$$
		\mathcal{H}^j(Z^\prime_{(\Psi, T)}) \leq \sum_{|I|=j} E \Vol_j(\pi_{C^I}(Z^\prime_{(\Psi, T)})) \leq K V_j(Z^\prime_{(\Psi, T)}) \leq K V_j(Z_T),
		$$
		where $K = \max_{ k \leq \mathcal{F}}\max_j {k \choose j} E = \poly_{\mathcal{F}}(D)$. 
		\\\\
		By Lemma \ref{lem:auxilliary_set}, 
		$$
		V^\prime_j(Z_T) \leq \sup_{\Phi \in O_n(\R)} \mathcal{H}^j(Z^\prime_{(\Phi, T)}) \leq  K V_j(Z_T).
		$$
	\end{proof}
	
	\subsection{Proof of Theorem \ref{thm:sharp_lattice_count}}
	\begin{proof}
		We consider the set $C = \{(T,x) \in \R^{n+m}: x \in \cl(Z_T)\}$. By Lemma \ref{lem:sharp_family_closure}, $U \in \Omega_{\bigO_\mathcal{F}(1), \poly_\mathcal{F}(D)}$. By Lemma \ref{lem:uniform_davenport_constant} there is some $M = \poly_\mathcal{F}(D)$ such that for all $T \in \R^m$ and endomorphisms $\Psi$, such that $|C_T \cap \Lambda| = |\Psi(C_T) \cap \Z^n|$,
		$$
		\left| \left| C_T \cap \Lambda \right| - \frac{\Vol(C_T)}{\det(\Lambda)} \right| \leq \sum^{n-1}_{j=0} M^{n-j}V_j(\Psi(C_T)).
		$$
		By Lemma \ref{lem:endomorphism_image_bound}, Lemma \ref{lem:orthogonal_projection_bound}, and Proposition \ref{prop:bound_V_prime} we have
		\begin{align*}
			\sum^{n-1}_{j=0} M^{n-j}V_j(\Psi(C_T)) &\leq  \sum^{n-1}_{j=0} M^{n-j}\frac{2^j}{B_j \lambda_1 \ldots \lambda_j}\sum_{|I| = j}\Vol_j(C_T^I)\\
			&\leq \sum^{n-1}_{j=0} M^{n-j}\frac{2^j}{B_j \lambda_1 \ldots \lambda_j}{n \choose j}\left(j^{3/2}\frac{n!2^n}{B_n}\right)^j V^\prime_j(C_T)\\
			&\leq \sum^{n-1}_{j=0} M^{n-j}\frac{2^j}{B_j \lambda_1 \ldots \lambda_j}{n \choose j}\left(j^{3/2}\frac{n!2^n}{B_n}\right)^j K V_j(C_T).
		\end{align*}
		
		Let
		$$
		c_C = \max_{k \leq \mathcal{F}} \max_{ j \leq k} M^{k-j}K\frac{2^j}{B_j}{k \choose j}\left(j^{3/2}\frac{k!2^k}{B_k}\right)^j.
		$$
		Note that $c_C = \poly_{\mathcal{F}}(D)$.
		By Lemma \ref{lem:volume_of_boundary} we have that $\Vol(\cl(Z_T)) = \Vol(Z_T)$ and by Lemma \ref{lem:Vj_of_closure}, $V_J(\cl(Z_T)) = V_j(Z_T)$. Thus we have
		$$
		\left| \left|\cl(Z_T) \cap \Lambda \right| - \frac{\Vol(Z_T)}{\det(\Lambda)} \right| \leq c_C \sum^{n-1}_{j=0} \frac{V_j(Z_T)}{\lambda_1 \ldots \lambda_j}
		$$
		and $c_C = \poly_\mathcal{F}(D)$.
		Similarly, for the set $P = \{(T,x) \in \R^{m+n}: x \in \partial(Z_T)\}$ we have a constant $c_P = \poly_\mathcal{F}(D)$ such that 
		$$
		\left| \left|\partial(Z_T) \cap \Lambda \right| - \frac{\Vol(\partial(Z_T))}{\det(\Lambda)} \right| \leq c_P \sum^{n-1}_{j=0} \frac{V_j(\partial(Z_T))}{\lambda_1 \ldots \lambda_j}.
		$$
		Since $\Vol(\partial(Z_T)) = 0$ and $\partial(Z_T) \subseteq \cl(Z_T)$ and thus $V_J(\partial(Z_T)) \leq V_j(\cl(Z_T)) = V_j(Z_T)$, we have
		$$
		\left|\partial(Z_T) \cap \Lambda \right| \leq c_P \sum^{n-1}_{j=0} \frac{V_j(Z_T)}{\lambda_1 \ldots \lambda_j}.
		$$
		Now 
		$$
		||Z_T \cap \Lambda| - |\cl(Z_T) \cap \Lambda|| \leq |\partial(Z_T) \cap \Lambda| \leq c_P \sum^{n-1}_{j=0} \frac{V_j(Z_T)}{\lambda_1 \ldots \lambda_j}.
		$$
		Therefore 
		\begin{align*}
			\left|\left|Z_T \cap \Lambda \right| - \frac{\Vol(Z_T)}{\det(\Lambda)} \right| &= \left|\left|Z_T \cap \Lambda \right| - |\cl(Z_T) \cap \Lambda| + |\cl(Z_T) \cap \Lambda|- \frac{\Vol(Z_T)}{\det(\Lambda)} \right|\\
			&\leq \left| \left|Z_T \cap \Lambda \right| - \left|\cl(Z_T) \cap \Lambda\right|\right| +  \left| \left|\cl(Z_T) \cap \Lambda \right| - \frac{\Vol(Z_T)}{\det(\Lambda)} \right|\\
			&\leq c_P \sum^{n-1}_{j=0} \frac{V_j(Z_T)}{\lambda_1 \ldots \lambda_j} + c_C \sum^{n-1}_{j=0} \frac{V_j(Z_T)}{\lambda_1 \ldots \lambda_j}\\
			&= c \sum^{n-1}_{j=0} \frac{V_j(Z_T)}{\lambda_1 \ldots \lambda_j},
		\end{align*}
		where $c = c_C + c_P = \poly_\mathcal{F}(D)$.
	\end{proof}
	
	\section{Effective Lattice Point Counting in $\R_{\exp}$}\label{sec: effec LPC Rexp}
	
	\subsection{Restricted Sub-Pfaffian Sets}
	In order to prove Theorem \ref{thm:Rexp_lattice_count}, we introduce the \#o-minimal structure $\R_{\text{rPfaff}}$ of restricted sub-Pfaffian sets. We then show that Theorem \ref{thm:presharp_to_sharp} applies uniformly to restrictions of existentially definable sets in $\R_{\exp}$. Since we require that the fibres $Z_T$ in Theorems \ref{thm:sharp_lattice_count} and \ref{thm:Rexp_lattice_count} are bounded, we can lift the bound obtained on restrictions of existentially definable sets to the full $\R_{\exp}$-definable set. This extension idea is used by Jones and Thomas in \cite{effective_pila_wilkie_pfaffian} to similarly extend a point counting result from restricted sub-Pfaffian sets to unrestricted sub-Pfaffian sets.
	
	\begin{definition}\label{def:pfaffian_function}
		Let $U\subseteq \R^n$ be a product of open intervals. A sequence $f_1, \ldots, f_k: U \to \R$ of analytic functions is called a \textit{Pfaffian chain} if there exist polynomials $P_{i,j} \in \R[X_1, \ldots, X_n, Y_1, \ldots, Y_i]$, for $i = 1, \ldots, k$ and $j = 1, \ldots, n$ such that
		$$
		\frac{\partial f_i}{\partial x_j}(x) = P_{i,j}(x, f_1(x), \ldots, f_i(x)),
		$$
		for all $i,j$ and $x \in U$. We say that a function $f$ is \textit{Pfaffian} with chain $f_1, \ldots, f_k$ if $f(x) = P(x, f_1(x), \ldots, f_k(x))$ for all $x \in U$, for some $P \in \R[X_1, \ldots, X_n, Y_1, \ldots, Y_n]$. The \textit{Pfaffian format} of a $f$ is defined as $n + k$ and its \textit{Pfaffian degree} is defined as $\sum_{i,j} \deg(P_{i,j}) + \deg(P)$.
	\end{definition}
	
	\noindent Polynomials, $p \in \R[X_1, \ldots, X_n]$, are naturally Pfaffian functions with format $n$ and degree $\deg(p)$. For any suitable $U\subset \R$, The function $\exp|_U: \R \to \R$ is Pfaffian with chain $\exp|_U$ and therefore has Pfaffian format $1$ and Pfaffian degree $2$. 
	
	\begin{definition}\label{def:semi_pfaffian_set}
		A \textit{semi-Pfaffian set} is a set $X \subseteq \R^n$, for $n \in \N$, such that the $X$ is defined by a Boolean combination of formulas of the form $f(x) = 0$ and inequalities of the form $h(x) > 0$ where all functions $f,h: U \to \R$ are Pfaffian and for some common $U \subset \R^n$ where $U$ is as in Definition \ref{def:pfaffian_function}. The \textit{Pfaffian format} of $X$ is defined as the maximum of the formats of the functions in its definition and its \textit{Pfaffian degree} is defined as the sum of their degrees. 
	\end{definition}
	
	\begin{definition}\label{def:sub_pfaffian_set}
		A set $Y \subseteq \R^n$ is called \textit{sub-Pfaffian} if there exists some semi-Pfaffian set $X \subseteq \R^{n + m}$ for some $m \in \N$ and projection map $\pi:\R^{n+m} \to \R^n$ such that $Y = \pi(X)$. The Pfaffian format and Pfaffian degree of $Y$ are those of $X$. 
	\end{definition}
	
	\noindent Let $f:U \to \R$ be a Pfaffian function. For any open box $B$ such that $\overline{B} \subseteq U$ we call $f|_B$ a \textit{restricted} Pfaffian function. We say that $f|_B$ has the same Pfaffian format and degree as $f$. We then define restricted semi-Pfaffian and restricted sub-Pfaffian sets as in Definitions \ref{def:semi_pfaffian_set} and \ref{def:sub_pfaffian_set} but using restricted Pfaffian functions in place of Pfaffian functions. 
	
	\begin{definition}[\cite{effective_subpfaffian_cell_decomposition} Definition 3]
		Let $Y \subseteq \R^n$ be a restricted sub-Pfaffian set. We say that $Y$ has \textit{*-format} $\mathcal{F}$ and \textit{*-degree} $D$ if there are finitely many restricted semi-Pfaffian sets $X_i \subseteq \R^{k_i}$ and connected components $X_i^\prime$ of the $X_i$ such that
		$$
		Y = \bigcup_i \pi_i(X^\prime_i),
		$$
		where $\pi_i: \R^{k_i} \to \R^n$ is the coordinate projection and the semi-Pfaffian sets $X_i$ have Pfaffian formats with maximum, $\mathcal{F}$, and Pfaffian degrees with sum $D$.
	\end{definition}
	
	\noindent In \cite{wilkies_conjecture}, Binyamini, Novikov and Zack show, using work of Binyamini and Vorobjov in \cite{effective_subpfaffian_cell_decomposition}, that the structure $\R_{\text{rPfaff}}$ of restricted sub-Pfaffian sets equipped with the FD-filtration $\Omega^*$ given by $X \in \Omega^*_{\mathcal{F},D}$ if and only if $X$ has *-format, $\mathcal{F}$, and *-degree $D$ is weakly \#o-minimal (or \#o-minimal in the terminology of \cite{wilkies_conjecture}) with \#cell decomposition. It follows from Theorem 1.9 in \cite{sharp_o_minimality} (Theorem \ref{thm:presharp_to_sharp} in this paper) that $\Omega^*$ can be extended to an FD-filtration, $\Omega$, such that $(\R_{\text{rPfaff}}, \Omega)$ is \#o-minimal with \#-cell decomposition in the updated language of \cite{sharp_o_minimality}. Furthermore, we have by Remark 10 in \cite{effective_subpfaffian_cell_decomposition} that a sub-Pfaffian set of format $\mathcal{F}$ and degree $D$ has *-format $\mathcal{F}$ and *-degree $\poly_\mathcal{F}(D)$.
	
	\subsection{Proof of Theorem \ref{thm:Rexp_lattice_count}}
	\begin{proof}
		Suppose that $\varphi$ is an $L_{\exp}$-formula of the form $\exists y_1 \ldots \exists y_n \psi(x, y)$ such that $\psi(x,y)$ is quantifier free and let $Y = \psi(\R)$. Then there is some projection map, $\pi$, such that $Z = \pi(Y)$.
		\\\\
		Let $\psi|_{[0,1]}$ denote the formula obtained from $\psi$ by replacing every instance of $\exp$ in $\psi$ with $\exp|_{[0,1]}$. Then $Y_0 = \psi(\R)$ and $Z_0 =\pi(Y_0)$ are definable in $\R_{\text{rPfaff}}$ and have Pfaffian format $\mathcal{F}$
		and degree $D$.        
		\\\\
		Let $M \in \N$. Then the set $Y_M := Y \cap [-M,M]^{n+m}$ can be obtained by replacing each instance of $\exp|_{[0,1]}$ in $\psi|_{[0,1]}$ with an appropriate restriction of $\exp$ to a compact subset of $\R$ has Pfaffian format $\mathcal{F}$ and degree $D$ since each restriction of $\exp$ has the same Pfaffian format and degree.
		\\\\
		In order to apply Theorem \ref{thm:sharp_lattice_count} we need a \#o-minimal structure with $\#$-cell decomposition, thus we must move from Pfaffian format and degree to *-format and *-degree. There is some $D^\prime = \poly_\mathcal{F}(D)$ such that each $Y_M$ has *-format $\mathcal{F}$ and *-degree $D^\prime$. Therefore $Z_M = \pi(Y_M)$ has *-format $\mathcal{F}$ and *-degree $D^\prime$.
		\\\\
		Applying Theorem \ref{thm:sharp_lattice_count} to each $Z_M$ in the \#o-minimal structure $(\R_{\text{rPfaff}}, \Omega)$ extending $(\R_{\text{rPfaff}}, \Omega^*)$, we get that there exists some $c = \poly_{F}(D^\prime) = \poly_\mathcal{F}(D)$, which does not depend on $M$ such that for all lattices $\Lambda$ with basis $\{\lambda_1, \ldots, \lambda_n\}$,
		$$
		\left| \left|(Z_M)_T \cap \Lambda\right| - \frac{\Vol((Z_M)_T)}{\det(\Lambda)} \right| \leq c \sum^{n-1}_{j=0} \frac{V_j((Z_M)_T)}{\lambda_1, \ldots, \lambda_j}.
		$$
		Now, since each fibre $Z_T$ is bounded, for each $T \in \R^m$ there exists some $M \in \N$ such that $Z_T = (Z_M)_T$.
	\end{proof}
	
	\section{A Note on Effective o-Minimality}\label{sec: eff omin}
	We conclude with a brief note on an alternative framework that yields an effective version of Theorem \ref{thm:BW_lattice_counting}. In a recent paper \cite{effective_o-minimality}, Binyamini introduces a new formalisation of effective o-minimality. In the vein of \#o-minimality, effective o-minimality uses a filtration on definable subsets of $\R$ with bounds on their complexity growth under boolean combination and projection. The growth constraints, however, are weaker than those for sharply o-minimal structures. In \cite{effective_pila_wilkie_pfaffian} Binyamini, Jones, Schmidt and Thomas give an effective version of Pila-Wilkie for sets definable using Pfaffian functions.
	
	\begin{definition}
		An \textit{effectively} o-minimal structure is an o-minimal structure on $\R$ expanding the real field, $\mathcal{R}$, with a filtration $\Omega_\mathcal{F}$ on the definable subsets of $\R^n$ for $n \in \N$ and a constant $c = \bigO_\mathcal{F}(1)$ such that:
		\begin{description}[labelindent=1cm]
			\item[E1] For all $\mathcal{F} \in \N$, we have $\Omega_\mathcal{F} \subseteq \Omega_{\mathcal{F} + 1}$ and all definable sets in $\R$ lie in some $\Omega_\mathcal{F}$ for some $\mathcal{F} \in \N$.
			\item[E2] For every $A,B \in \R^n$ such that $A, B \in \Omega_{\mathcal{F}}$, we have
			$$ A \cup B, A \cap B, \R^n \setminus A, A \times B, \pi^n_k(A) \in \Omega_{\mathcal{F} + 1},$$
			where $\pi^n_k: \R^n \to \R^k$ is the projection map to the first $k$ coordinates.
			\item[E3] If $A \subseteq \R$ and $A \in \Omega_{\mathcal{F}}$ then $A$ has at most $c$ connected components.
		\end{description}
		
		\noindent Binyamini makes explicit the link between effective o-minimality and \#o-minimality in \cite{effective_o-minimality}. Let $(\mathcal{R}, \Omega)$ be a \#o-minimal expansion of the real field. We can (up to minor reindexing) obtain an effectively o-minimal structure on $\mathcal{R}$ by considering the filtration $\Omega^\prime$ such that for all $X \in \Omega_{\mathcal{F}, D}$ we have $X \in \Omega^\prime_{\mathcal{F}}$. This axiom scheme yields effective versions of many of the standard theorems in o-minimality. Of interest to us are effective cell-decomposition and effective definable choice \cite[Section 1.5]{effective_o-minimality}.
		\begin{theorem}
			Let $(\mathcal{R}, \Omega)$ be an effectively o-minimal structure and let $X_1, \ldots,\\ X_k \subseteq \R^n$ be definable sets such that $X_1, \ldots, X_k \in \Omega_{\mathcal{F}}$. Then there exists a cylindrical decomposition of $\R^n$ compatible with $X_1, \ldots, X_n$ such that the number of cells and format of each cell are bounded by some constant $c_\mathcal{F} = \bigO_\mathcal{F}(1)$.
		\end{theorem}
		
		\begin{theorem}
			Let $(\mathcal{R}, \Omega)$ be an effectively o-minimal structure and let $Z \subseteq \R^{m+n}$ be a definable family such that $Z \in \Omega_{\mathcal{F}}$ and for all $T \in \R^m$ the fibre $Z_T$ is non-empty. Let $\pi_m: \R^{m+n} \to \R^m$ denote the projection map to the first $m$-coordinates. Then there exists some definable function $f: \pi_m(Z) \to \R^n$ such that for all $T \in \pi_m(Z)$, the image $f(T)$ is in $Z_T$ and $\Gamma_f \in \Omega_{\bigO_\mathcal{F}(1)}$.
		\end{theorem}
		
		\noindent One can obtain analogues of Lemmas \ref{lem:sharp_family_closure}, \ref{lem:uniform_davenport_constant}, \ref{lem:sharp_bound_on_Hj} and \ref{lem:auxilliary_set} as well as Proposition \ref{prop:bound_V_prime} by replacing each instance of $\Omega_{\bigO_\mathcal{F}(1),\bigO_\mathcal{F}(1)}$ and $\Omega_{\bigO_\mathcal{F}(1),\poly_\mathcal{F}(D)}$ with $\Omega_{\bigO_\mathcal{F}(1)}$ and every instance of $\poly_{\mathcal{F}}(D)$ with $\bigO_{\mathcal{F}}(1)$. It is easy to see that all the necessary auxiliary sets in those proofs can be shown to have format $\bigO_{\mathcal{F}}(1)$ (where $Z$ has format $\mathcal{F}$) by applying the axioms of effective o-minimality as you would with the analogous axioms of \#o-minimality. The following analogous version of Theorem \ref{thm:sharp_lattice_count} follows from there.
		
		\begin{theorem}
			Let $Z \subseteq \R^{m+n}$, be a definable family in some effectively o-minimal structure such that $Z \in \Omega_{\mathcal{F}}$ for some $\mathcal{F} \in \N$ and the fibre $Z_T \subseteq \R^n$ is bounded for each $T \in \R^m$. Then there exists some constant $c = \bigO_{\mathcal{F}}(1)$ such that for all $T \in \R^m$
			$$
			\left| \left|Z_T \cap \Lambda\right| - \frac{\Vol(Z_T)}{\det(\Lambda)} \right| \leq c \sum^{n-1}_{j=0} \frac{V_j(Z_T)}{\lambda_1 \ldots \lambda_j},
			$$
			where $V_j(Z)$ is the sum of volumes of the $j$-dimensional orthogonal projections of $Z$ onto the coordinate spaces obtained by setting $n-j$ coordinates to zero and $V_0(Z)$ is taken to be $1$ and $\lambda_1, \ldots, \lambda_n$ are the successive minima of $\Lambda$ with respect to a zero centred unit ball.   
		\end{theorem}
	\end{definition}
	
	\begin{ack}
		The first author is grateful to the University of Manchester's Faculty of Science and Engineering for support. The second author is grateful to the Heilbronn Institute for Mathematical Research for support. The authors would also like to thank Gareth Jones for his various helpful comments on this work.
	\end{ack}

	\printbibliography
\end{document}